\documentclass[11pt]{article}
\usepackage[a4paper,margin=1in]{geometry}
\usepackage[T1]{fontenc}
\usepackage[utf8]{inputenc}
\usepackage{lmodern}
\usepackage{amsmath,amssymb,amsthm,mathtools,bm}
\usepackage{microtype}
\usepackage{enumitem}
\usepackage{booktabs}
\usepackage{xcolor}
\usepackage{hyperref}
\usepackage{aliascnt}
\usepackage[nameinlink,capitalize]{cleveref}
\hypersetup{colorlinks=true,linkcolor=blue!55!black,citecolor=blue!55!black,urlcolor=blue!55!black}

\numberwithin{equation}{section}
\newtheorem{theorem}{Theorem}[section]

\newaliascnt{proposition}{theorem}
\newtheorem{proposition}[proposition]{Proposition}
\aliascntresetthe{proposition}

\newaliascnt{lemma}{theorem}
\newtheorem{lemma}[lemma]{Lemma}
\aliascntresetthe{lemma}

\newaliascnt{corollary}{theorem}
\newtheorem{corollary}[corollary]{Corollary}
\aliascntresetthe{corollary}

\theoremstyle{definition}
\newaliascnt{remark}{theorem}
\newtheorem{remark}[remark]{Remark}
\aliascntresetthe{remark}

\newaliascnt{definition}{theorem}
\newtheorem{definition}[definition]{Definition}
\aliascntresetthe{definition}

\crefname{theorem}{Theorem}{Theorems}
\crefname{proposition}{Proposition}{Propositions}
\crefname{lemma}{Lemma}{Lemmas}
\crefname{corollary}{Corollary}{Corollaries}
\crefname{remark}{Remark}{Remarks}
\crefname{definition}{Definition}{Definitions}

\newcommand{\R}{\mathbb R}
\newcommand{\dd}{\,d}
\newcommand{\Gam}{\Gamma}
\newcommand{\lapfive}{\Delta_5}

\newcommand{\norm}[1]{\left\lVert #1\right\rVert}

\title{\textbf{Critical Compression in Axisymmetric Swirl}}
\author{Rishad Shahmurov\\
Cellular Products Research and Development, Roswell, Georgia, USA\\
\texttt{rshahmurov@crimson.ua.edu}}
\date{September 2026}

\begin{document}
\maketitle

\begin{abstract}
We study two scale-critical inward-compression mechanisms in smooth axisymmetric Navier--Stokes flow with swirl.  For the weighted swirl family $Y_d=r^d u^\theta$, $-1\le d<1$, an exact $L^p$ balance yields a signed radial-compression action.  If $p\ge3/(1-d)$, finite $L^p$ mass and a bounded signed action imply continuation.  On the critical diagonal $Y_q=r^{1-3/q}u^\theta$, $q\ge3/2$, the integrating factor is universal, $\exp(3\int\overline U_q)$; $H^1$ data automatically supply the critical mass for $2\le q\le3$, while bounded circulation extends this to every finite $q>3$.  Independently, a sliding Petrovskii displacement criterion reaches the endpoint coefficient $2$ and admits a quantified iterated-log correction, while a pulse train shows terminal-only control is insufficient.  For profiles with the corresponding weighted moments, the instantaneous Hodge response to $\partial_z(F^2)$ is non-amplifying for power-weighted compressive $F^2$-work throughout $1\le\alpha\le5$.  Thus a finite-time singularity must escape a continuum of critical signed-compression detectors and recurrently cross the Petrovskii wall.
\end{abstract}



\section{Introduction}

We consider the incompressible Navier--Stokes equations
\begin{equation}\label{eq:NS}
 \partial_t u+(u\cdot\nabla)u+\nabla p=\nu\Delta u,
 \qquad \nabla\cdot u=0,
 \qquad x\in\R^3,
\end{equation}
with $\nu>0$ and smooth axisymmetric data in $H^2(\R^3)$.  In cylindrical coordinates,
\[
 u=u^r(r,z,t)e_r+u^\theta(r,z,t)e_\theta+u^z(r,z,t)e_z.
\]
The no-swirl class is globally regular \cite{UkhovskiiYudovich1968,Ladyzhenskaya1968}, whereas arbitrary swirl remains open.  Existing restrictions include scale-critical one-component criteria \cite{ChaeLee2002}, blowup-rate estimates \cite{ChenStrainTsaiYau2008,ChenStrainTsaiYau2009}, circulation criteria \cite{LeiZhang2017,Wei2016}, angular-velocity and vorticity criteria \cite{ChenFangZhang2017,Zhang2016}, weighted component criteria \cite{RenclawowiczZajaczkowski2019}, and quantitative critical-space bounds \cite{OzanskiPalasek2023}.  Zhang recently proved regularity under the one-sided partial type-I condition $u^r\gtrsim-(T-t)^{-1/2}$, assuming bounded initial circulation \cite{Zhang2026}.

Write
\begin{equation}\label{eq:variables}
 \Gamma=ru^\theta,
 \qquad
 F=\frac{u^\theta}{r},
 \qquad
 G=\frac{\omega^\theta}{r},
 \qquad
 U=\frac{u^r}{r},
\end{equation}
and suppress the harmless angular factor $2\pi$, so $\dd\mu_3=r\,\dd r\,\dd z$.  The first part of the paper concerns the one-parameter family
\begin{equation}\label{eq:Yd-intro}
 Y_d=r^d u^\theta,
 \qquad -1\le d<1.
\end{equation}
The endpoints are $Y_{-1}=F$ and, formally as $d\uparrow1$, $Y_1=\Gamma$.  For $-1<d<1$ and $p>1$ we prove the exact balance
\begin{align}\label{eq:intro-Ydp}
 \frac1p\frac{\dd}{\dd t}\int |Y_d|^p\,\dd\mu_3
 &+\nu\frac{4(p-1)}{p^2}\int\bigl|\nabla |Y_d|^{p/2}\bigr|^2\,\dd\mu_3
 +\nu(1-d^2)\int\frac{|Y_d|^p}{r^2}\,\dd\mu_3 \\
 & =-(1-d)\int U|Y_d|^p\,\dd\mu_3.
\end{align}
At $d=-1$ the Hardy-potential term is replaced by the positive axis trace already present in the $F$ equation.  Define the active-strain average
\begin{equation}\label{eq:intro-Ubar-dp}
 \overline U_{d,p}(t)
 =\frac{\int U|Y_d|^p\,\dd\mu_3}{\int|Y_d|^p\,\dd\mu_3},
\end{equation}
with value zero when the denominator vanishes.  The exact integrating factor in \eqref{eq:intro-Ydp} shows that if
\begin{equation}\label{eq:intro-action-dp}
 \sup_{t_0<t<T}\int_{t_0}^t-\overline U_{d,p}(s)\,\dd s<\infty,
\end{equation}
then $|Y_d|^{p/2}\in L_t^\infty L_x^2\cap L_t^2\dot H_x^1$ and hence
\[
 Y_d\in L_t^{4p/3}L_x^{2p}.
\]
Combining this with the weighted angular-velocity criteria of Chen--Fang--Zhang \cite{ChenFangZhang2017} gives continuation whenever
\begin{equation}\label{eq:intro-dp-threshold}
 p\ge\frac{3}{1-d}.
\end{equation}
The weighted energy itself is closely related to earlier weighted-swirl estimates, including the critical family used by Renc\l awowicz--Zaj\k aczkowski \cite{RenclawowiczZajaczkowski2019}.  The point here is different: we retain the \emph{signed} strain average in the exact integrating factor, so radial expansion is allowed to cancel earlier compression rather than being replaced by an unsigned norm.

The scale-critical diagonal in \eqref{eq:intro-dp-threshold} has a particularly simple form.  For $3/2\le q<\infty$ set
\begin{equation}\label{eq:Yq-intro}
 d_q=1-\frac3q,
 \qquad
 Y_q=r^{1-3/q}u^\theta,
 \qquad
 \overline U_q=\frac{\int U|Y_q|^q\,\dd\mu_3}{\int|Y_q|^q\,\dd\mu_3}.
\end{equation}
Then $\|Y_q\|_q$ is invariant under Navier--Stokes scaling and the integrating factor is independent of $q$:
\begin{equation}\label{eq:intro-universal-factor}
 \exp\left(3\int_{t_0}^t\overline U_q(s)\,\dd s\right).
\end{equation}
For smooth $H^1$ data, $Y_q(t_0)\in L^q$ automatically for $2\le q\le3$; bounded circulation extends this to every finite $q>3$.  In particular, in the bounded-circulation class a hypothetical first singularity must make the signed action diverge for \emph{every} $q\ge2$.  The case $q=3$ is especially elementary: $Y_3=u^\theta$, and its critical mass is already finite by $H^1\hookrightarrow L^6$ and interpolation.  Thus the new family contains a scale-critical action criterion using the physical angular velocity itself.

The second result is geometric.  Let
\begin{equation}\label{eq:mDintro}
 m(t)=\norm{(u^r)_-(t)}_{L^\infty(\R^3)},
 \qquad
 D_\sigma(h)=\int_{\sigma-h}^{\sigma}m(t)\,\dd t,
\end{equation}
and, for $0<h<h_0$,
\begin{equation}\label{eq:ell}
 \ell_{h_0}(h)=\log\log\left(e^e\frac{h_0}{h}\right).
\end{equation}
A moving-boundary heat estimate yields regularity whenever the sliding displacement obeys
\begin{equation}\label{eq:intro-petro}
 D_\sigma(h)\le2\sqrt{\nu h\,\ell_{h_0}(h)}
\end{equation}
uniformly on sufficiently small windows.  The endpoint coefficient $2$ is included.  The more precise integral estimate also permits
\begin{equation}\label{eq:intro-petro-secondary}
 \frac{D_\sigma(h)^2}{4\nu h}
 \le \ell_{h_0}(h)+\beta\log\!\bigl(e+\ell_{h_0}(h)\bigr)+C_0,
 \qquad \beta<\frac12.
\end{equation}
The induced circulation modulus is stronger than Wei's logarithmic regularity threshold \cite{Wei2016}.  A smooth pulse train shows that the same scale of control imposed only on intervals ending at the final time does not prevent arbitrarily thin preterminal layers; this comparison mechanism genuinely requires sliding control.

Finally, the source term in the $G$ equation has a definite instantaneous sign after Hodge recovery.  Freezing $F$ and retaining only $\dot G_{\rm src}=\partial_z(F^2)$, we prove that the corresponding variation of $U$ is non-amplifying for the power-weighted compressive $F^2$ functional throughout the Hardy-nonnegative range $1\le\alpha\le5$.  Consequently, a hypothetical singular endpoint must simultaneously escape a continuum of signed scale-critical compression detectors and the sliding Petrovskii regime, while any positive self-source contribution to compression must involve temporal evolution beyond the instantaneous source leg.

\section{Axisymmetric identities}

The physical incompressibility constraint is
\begin{equation}\label{eq:divphys}
 \partial_r(ru^r)+\partial_z(ru^z)=0.
\end{equation}
The variables in \eqref{eq:variables} satisfy
\begin{align}
 \partial_t\Gam+u^r\partial_r\Gam+u^z\partial_z\Gam
 &=\nu\left(\partial_{rr}-\frac1r\partial_r+\partial_{zz}\right)\Gam,
 \label{eq:Gamma}\\
 \partial_tF+u^r\partial_rF+u^z\partial_zF
 &=\nu\lapfive F-2UF,
 \label{eq:F}\\
 \partial_tG+u^r\partial_rG+u^z\partial_zG
 &=\nu\lapfive G+\partial_z(F^2).
 \label{eq:G}
\end{align}
For the Hodge recovery, let $\phi$ solve
\begin{equation}\label{eq:Hodge}
 -\lapfive\phi=G,
 \qquad
 u^r=-r\partial_z\phi,
 \qquad
 u^z=2\phi+r\partial_r\phi.
\end{equation}
Thus
\begin{equation}\label{eq:U-Hodge}
 U=-\partial_z\phi=-\partial_z(-\lapfive)^{-1}G.
\end{equation}
All formulas below are first proved for smooth axis-compatible functions with sufficient decay; standard truncation and approximation remove the auxiliary decay assumptions.

For $-1\le d\le1$ define
\begin{equation}\label{eq:Yd}
 Y_d=r^d u^\theta.
\end{equation}
A direct calculation from the $u^\theta$ equation gives, for $-1<d<1$,
\begin{equation}\label{eq:Yd-PDE}
 \partial_tY_d+u^r\partial_rY_d+u^z\partial_zY_d
 =\nu\left(
 \partial_{rr}Y_d+\frac{1-2d}{r}\partial_rY_d+\partial_{zz}Y_d
 +\frac{d^2-1}{r^2}Y_d\right)
 -(1-d)UY_d.
\end{equation}
The same formula is valid algebraically at $d=-1$ and $d=1$, but the integration by parts at $d=-1$ retains the nonzero axis trace of $F=Y_{-1}$.

\begin{proposition}\label{prop:Ydp}
Let $p>1$ and $0<t_0<t_1<T$ lie in an interval of smooth existence.
\begin{enumerate}[label=\rm(\alph*)]
\item If $-1<d<1$ and $Y_d(t_0)\in L^p(\R^3)$, then
\begin{align}\label{eq:Ydp-energy}
 \frac1p\frac{\dd}{\dd t}\int |Y_d|^p\,\dd\mu_3
 &+\nu\frac{4(p-1)}{p^2}\int\bigl|\nabla |Y_d|^{p/2}\bigr|^2\,\dd\mu_3
 +\nu(1-d^2)\int\frac{|Y_d|^p}{r^2}\,\dd\mu_3 \\
 &=-(1-d)\int U|Y_d|^p\,\dd\mu_3.
\end{align}
\item If $d=-1$, $Y_{-1}=F$, and $F(t_0)\in L^p(\R^3)$, then
\begin{align}\label{eq:Fp}
 \frac1p\frac{\dd}{\dd t}\int |F|^p\,\dd\mu_3
 &+\nu\frac{4(p-1)}{p^2}\int\bigl|\nabla |F|^{p/2}\bigr|^2\,\dd\mu_3
 +\frac{2\nu}{p}\int_\R |F(0,z,t)|^p\,\dd z \\
 &=-2\int U|F|^p\,\dd\mu_3.
\end{align}
\item If $d=1$, $Y_1=\Gamma$, and $\Gamma(t_0)\in L^p(\R^3)$, then
\begin{equation}\label{eq:Gamma-p}
 \frac1p\frac{\dd}{\dd t}\int |\Gamma|^p\,\dd\mu_3
 +\nu\frac{4(p-1)}{p^2}\int\bigl|\nabla |\Gamma|^{p/2}\bigr|^2\,\dd\mu_3=0.
\end{equation}
\end{enumerate}
\end{proposition}

\begin{proof}
We give the calculation because the signs and the endpoint $d=-1$ are essential.  For $p\ge2$, multiply \eqref{eq:Yd-PDE} by $|Y_d|^{p-2}Y_d$ and integrate against $r\,\dd r\,\dd z$.  The transport term vanishes by \eqref{eq:divphys}.  Put $Y=Y_d$.  The axial diffusion contributes
\[
 -(p-1)\int |Y|^{p-2}|Y_z|^2\,\dd\mu_3.
\]
For the radial differential part, on $r>0$,
\begin{align*}
 &\int_0^\infty
 \left(Y_{rr}+\frac{1-2d}{r}Y_r\right)|Y|^{p-2}Y\,r\,\dd r\\
 &\quad=-(p-1)\int_0^\infty |Y|^{p-2}|Y_r|^2r\,\dd r
 -2d\int_0^\infty Y_r|Y|^{p-2}Y\,\dd r.
\end{align*}
If $-1<d\le1$, smooth axisymmetry gives $u^\theta=O(r)$ and hence $Y_d=O(r^{1+d})$; the last integral has zero boundary contribution.  The potential $(d^2-1)Y/r^2$ therefore gives precisely the third term on the left of \eqref{eq:Ydp-energy}.  For $d=1$ both this term and the stretching coefficient vanish, yielding \eqref{eq:Gamma-p}.

At $d=-1$, $Y=F$ may have a nonzero finite axis trace.  Now the coefficient of the last radial integral is $2$, and
\[
 2\int_0^\infty F_r|F|^{p-2}F\,\dd r
 =-\frac2p|F(0,z,t)|^p,
\]
which gives the positive trace term in \eqref{eq:Fp} after multiplication by $\nu$ and transfer to the left.

For $1<p<2$, replace $|s|^p/p$ by the even convex regularization
\[
 \Phi_\varepsilon(s)
 =\frac{(s^2+\varepsilon^2)^{p/2}-\varepsilon^p}{p},
 \qquad
 \Psi_\varepsilon=\Phi_\varepsilon',
\]
for which
\[
 \Phi_\varepsilon''(s)
 =(s^2+\varepsilon^2)^{(p-4)/2}
 \bigl(\varepsilon^2+(p-1)s^2\bigr)\ge0.
\]
Test the equation by $\Psi_\varepsilon(Y_d)$ on $\{r>\rho\}$, perform the preceding integrations by parts, and then let $\rho\downarrow0$.  For $d>-1$, $Y_d=O(r^{1+d})$ kills the axis boundary; for $d=-1$ the radial boundary converges to $2\nu\int\Phi_\varepsilon(F(0,z,t))\,\dd z$.  On every compact preterminal interval $U$ is bounded, while
\[
 0\le \Phi_\varepsilon(s)\le\frac{|s|^p}{p},
 \qquad
 0\le s\Psi_\varepsilon(s)\le |s|^p,
\]
and $s\Psi_\varepsilon(s)\le C_p\Phi_\varepsilon(s)$.  Gronwall therefore gives an $\varepsilon$-independent mass bound and the corresponding regularized dissipation bound.  Monotone/dominated convergence recovers the mass, potential, stretching, and axis-trace terms.  For the gradient term, Fatou gives the limiting weighted dissipation and, since $\Phi_\varepsilon''(s)\le C_p|s|^{p-2}$ for $s\ne0$ while $\nabla Y=0$ almost everywhere on $\{Y=0\}$, dominated convergence gives equality.  This proves all three identities for $p>1$.
\end{proof}

\section{Scale-critical signed compression actions}

For $-1\le d<1$ and $p>1$, define
\begin{equation}\label{eq:Zdp-Udp}
 Z_{d,p}(t)=\int |Y_d|^p\,\dd\mu_3,
 \qquad
 \overline U_{d,p}(t)
 =\frac{\int U|Y_d|^p\,\dd\mu_3}{Z_{d,p}(t)},
\end{equation}
with $\overline U_{d,p}=0$ when $Z_{d,p}=0$.

\begin{theorem}\label{thm:two-parameter-action}
Let $u$ be a smooth axisymmetric solution of \eqref{eq:NS} on $[0,T)$ with $H^2(\R^3)$ initial data.  Let
\[
 -1\le d<1,
 \qquad
 p\ge\frac3{1-d}.
\]
When $-1\le d<0$, assume in addition that the initial circulation $\Gamma_0=ru_0^\theta$ belongs to $L^\infty$.
Assume that for some $0<t_0<T$,
\begin{equation}\label{eq:action-dp-hyp}
 Y_d(t_0)\in L^p(\R^3),
 \qquad
 \mathcal C_{d,p}(t_0,T)
 :=\sup_{t_0<t<T}\int_{t_0}^t-\overline U_{d,p}(s)\,\dd s<\infty.
\end{equation}
Then $u$ extends smoothly through $T$.
\end{theorem}

\begin{proof}
If $Z_{d,p}(t_0)=0$, then $u^\theta(t_0)\equiv0$, the no-swirl class is preserved, and classical no-swirl regularity applies.  Assume $Z_{d,p}(t_0)>0$.  On each compact interval $[t_0,T']\Subset[t_0,T)$, smoothness makes $U$ bounded, so Proposition~\ref{prop:Ydp} and the preceding regularization propagate the finite $L^p$ mass.

For $-1<d<1$, multiply \eqref{eq:Ydp-energy} by $p$ and set
\[
 B_{d,p}(t)=p(1-d)\int_{t_0}^t\overline U_{d,p}(s)\,\dd s.
\]
The balance becomes
\begin{align}\label{eq:dp-integrating-factor}
 &\frac{\dd}{\dd t}\left(e^{B_{d,p}(t)}Z_{d,p}(t)\right)
 +\frac{4\nu(p-1)}p e^{B_{d,p}(t)}
   \int\bigl|\nabla |Y_d|^{p/2}\bigr|^2\,\dd\mu_3\\
 &\hspace{28mm}
 +p\nu(1-d^2)e^{B_{d,p}(t)}
   \int\frac{|Y_d|^p}{r^2}\,\dd\mu_3=0.
\end{align}
At $d=-1$ the same formula holds with the Hardy-potential term replaced by
\[
 2\nu e^{B_{-1,p}(t)}\int_\R|F(0,z,t)|^p\,\dd z.
\]
The signed-action hypothesis implies
\[
 B_{d,p}(t)\ge-p(1-d)\mathcal C_{d,p}(t_0,T),
\]
so \eqref{eq:dp-integrating-factor} controls both
\[
 |Y_d|^{p/2}\in L^\infty(t_0,T;L^2(\R^3))
 \quad\hbox{and}\quad
 |Y_d|^{p/2}\in L^2(t_0,T;\dot H^1(\R^3)).
\]
Gagliardo--Nirenberg therefore gives
\begin{equation}\label{eq:Yd-mixed}
 Y_d\in L^{4p/3}(t_0,T;L^{2p}(\R^3)).
\end{equation}
For the mixed exponents $P=4p/3$ and $Q=2p$,
\begin{equation}\label{eq:mixed-line-d}
 \frac2P+\frac3Q=\frac3p\le1-d,
 \qquad
 Q=2p>\frac3{1-d}.
\end{equation}
Restart the classical solution at $t_0$.  The estimate \eqref{eq:Yd-mixed} therefore supplies the required mixed norm on the whole shifted interval $(t_0,T)$; no information on $(0,t_0)$ is used.  If $0\le d<1$, Theorem~1.1 of Chen--Fang--Zhang \cite{ChenFangZhang2017} applies directly to $r^d u^\theta=Y_d$.  If $-1\le d<0$, write $a=-d\in(0,1]$; Remark~1.3 of the same paper applies to $u^\theta/r^a=Y_d$ under the bounded-circulation hypothesis.  Thus \eqref{eq:Yd-mixed}--\eqref{eq:mixed-line-d} imply continuation in the whole range $-1\le d<1$.
\end{proof}

\begin{remark}\label{rem:known-weighted-context}
The weighted variables themselves are not new.  Chen--Fang--Zhang formulate regularity criteria for $r^d u^\theta$ and, by interpolation with bounded circulation, for $u^\theta/r^a$ \cite{ChenFangZhang2017}.  Renc\l awowicz--Zaj\k aczkowski derive weighted swirl energies for $u_\alpha=\Gamma/r^\alpha$ and select the scale-balanced relation $\alpha s=3$ \cite{RenclawowiczZajaczkowski2019}.  The role of \cref{thm:two-parameter-action} is to retain the exact signed strain average in the integrating factor.  The criterion is therefore not an unsigned control of $U$ or of a weighted swirl norm: radial expansion can compensate previous compression exactly as permitted by the evolution identity.
\end{remark}

\begin{remark}\label{rem:paper3-dictionary}
The same identity can be read directly in the power-weight coordinates of \cite{ShahmurovHardy2026}.  Since $Y_d=r^{d+1}F$ and $\dd\mu_3=r\,\dd r\,\dd z$,
\[
 \int |Y_d|^p\,\dd\mu_3
 =\int |F|^p r^{\alpha}\,\dd r\,\dd z,
 \qquad
 \alpha=p(d+1)+1.
\]
Consequently
\[
 \frac{\alpha-(2p+1)}p=d-1=-(1-d),
\]
so the stretching coefficient in the power-weight spectrum is exactly the coefficient in \cref{eq:Ydp-energy}.  Under the same conjugation, the radial Hardy and gradient terms combine into the positive potential $1-d^2$.  Thus the present two-parameter action family is the continuation-theoretic form of the power-weight spectrum, rather than an independent weighted identity.
\end{remark}

The critical line admits a particularly transparent parametrization.

\begin{definition}\label{def:critical-family}
For $3/2\le q<\infty$, set
\begin{equation}\label{eq:critical-Yq}
 d_q=1-\frac3q,
 \qquad
 Y_q=r^{1-3/q}u^\theta,
 \qquad
 Z_q=\int|Y_q|^q\,\dd\mu_3,
 \qquad
 \overline U_q=\frac{\int U|Y_q|^q\,\dd\mu_3}{Z_q},
\end{equation}
with $\overline U_q=0$ when $Z_q=0$.
\end{definition}

Under Navier--Stokes scaling $u_\lambda(x,t)=\lambda u(\lambda x,\lambda^2t)$,
\[
 Y_{q,\lambda}(x,t)=\lambda^{3/q}Y_q(\lambda x,\lambda^2t),
\]
so $\|Y_q\|_{L^q}$ is scale invariant.  Moreover,
\[
 q(1-d_q)=3,
\]
which removes $q$ from the integrating factor.

\begin{corollary}\label{cor:critical-family}
Let $u$ be a smooth axisymmetric $H^2$ solution on $[0,T)$ and fix $3/2\le q<\infty$.  When $3/2\le q<3$, assume bounded initial circulation $\Gamma_0\in L^\infty$.  If for some $t_0<T$
\begin{equation}\label{eq:critical-action-hyp}
 Y_q(t_0)\in L^q(\R^3),
 \qquad
 \mathcal A_q(t_0,T)
 :=\sup_{t_0<t<T}\int_{t_0}^t-\overline U_q(s)\,\dd s<\infty,
\end{equation}
then $u$ extends smoothly through $T$.
For $q>3/2$, the exact critical balance is
\begin{align}\label{eq:critical-balance}
 \frac{\dd}{\dd t}Z_q
 &+\frac{4\nu(q-1)}q
   \int\bigl|\nabla |Y_q|^{q/2}\bigr|^2\,\dd\mu_3
 +3\nu\left(2-\frac3q\right)
   \int\frac{|Y_q|^q}{r^2}\,\dd\mu_3\\
 &=-3\overline U_q Z_q.
\end{align}
At $q=3/2$, $Y_q=F$, the third term is replaced by
\[
 2\nu\int_\R|F(0,z,t)|^{3/2}\,\dd z.
\]
In every case the integrating factor is
\begin{equation}\label{eq:universal-critical-factor}
 \exp\left(3\int_{t_0}^t\overline U_q(s)\,\dd s\right).
\end{equation}
\end{corollary}

\begin{proof}
Take $d=d_q$ and $p=q$ in \cref{thm:two-parameter-action}.  Then
$3/(1-d_q)=q$, so the continuation condition is exactly critical.  Multiplying the identity in \cref{prop:Ydp} by $q$ and using
\[
 q(1-d_q)=3,
 \qquad
 q(1-d_q^2)=3(1+d_q)=3\left(2-\frac3q\right)
\]
gives \eqref{eq:critical-balance} and \eqref{eq:universal-critical-factor}.  The endpoint $q=3/2$ is precisely $d_q=-1$, hence follows from \eqref{eq:Fp}.
\end{proof}

\begin{proposition}\label{prop:automatic-critical-mass}
Let $u(t)$ be a smooth axisymmetric $H^1(\R^3)$ velocity field.  Then
\begin{equation}\label{eq:auto-Yq-basic}
 Y_q(t)=r^{1-3/q}u^\theta(t)\in L^q(\R^3)
 \qquad\text{for every }2\le q\le3.
\end{equation}
If in addition $\Gamma(t)=ru^\theta(t)\in L^\infty$, then
\begin{equation}\label{eq:auto-Yq}
 Y_q(t)\in L^q(\R^3)
 \qquad\text{for every }2\le q<\infty.
\end{equation}
Consequently, in the bounded-circulation strong-solution class the critical signed-action criterion \eqref{eq:critical-action-hyp} is automatically available at every finite $q\ge2$.
\end{proposition}

\begin{proof}
For $q=2$,
\[
 \|Y_2\|_2^2
 =2\pi\int_\R\int_0^\infty |u^\theta(r,z)|^2\,\dd r\,\dd z.
\]
Split at $r=1$.  On $r\ge1$ this is bounded by the kinetic energy.  On $0<r\le1$,
\[
 \int_{r\le1}|u^\theta|^2\,\dd r\,\dd z
 \le\int_{r\le1}\frac{|u^\theta|^2}{r}\,\dd r\,\dd z,
\]
and the right side is, up to $2\pi$, $\|u^\theta/r\|_{L^2(\R^3)}^2$.  Indeed, for the swirl vector $u^\theta e_\theta$,
\[
 |\nabla(u^\theta e_\theta)|^2
 =|\partial_r u^\theta|^2+|\partial_z u^\theta|^2+\frac{|u^\theta|^2}{r^2},
\]
so this term is contained in the ordinary $H^1(\R^3)$ norm.  Hence $Y_2\in L^2$.

For $q=3$, $Y_3=u^\theta\in L^3$ by interpolation between $L^2$ and $L^6$.  If $2<q<3$, put $a=6/q-2\in(0,1)$.  Pointwise,
\[
 |Y_q|=|Y_2|^a|Y_3|^{1-a},
 \qquad
 \frac1q=\frac a2+\frac{1-a}{3},
\]
so H\"older interpolation gives $Y_q\in L^q$.  This proves \eqref{eq:auto-Yq-basic}.

For $q\ge3$, if $\Gamma\in L^\infty$, the identity
\begin{equation}\label{eq:Yq-Gamma-interp}
 |Y_q|^q=|u^\theta|^3|\Gamma|^{q-3}
\end{equation}
gives
\[
 \|Y_q\|_q^q
 \le\|\Gamma\|_\infty^{q-3}\|u^\theta\|_3^3<\infty.
\]
This proves \eqref{eq:auto-Yq}.
\end{proof}

The circulation maximum principle gives
\[
 \|\Gamma(t)\|_{L^\infty}
 \le \|\Gamma_0\|_{L^\infty}
 \qquad (t<T),
\]
so under bounded initial circulation the conclusion of \cref{prop:automatic-critical-mass} is available at every smooth preterminal time.

\begin{proposition}\label{prop:q2-sharpness}
The lower endpoint $q=2$ in \cref{prop:automatic-critical-mass} is sharp for automatic control from the standard Sobolev class, even with bounded circulation.  More precisely, there exist smooth divergence-free axisymmetric data $u_0\in H^2(\R^3)$ with $\Gamma_0=ru_0^\theta\in L^\infty$ such that
\[
 Y_q(0)=r^{1-3/q}u_0^\theta\notin L^q(\R^3)
 \qquad\text{for every }\frac32\le q<2,
\]
whereas $Y_2(0)\in L^2(\R^3)$.
\end{proposition}

\begin{proof}
Choose $\chi\in C_c^\infty([0,\infty))$, with $\chi(r)=\chi_0(r^2)$ near the axis, and let
\[
 f(z)=(1+z^2)^{-1/4}\,[\log(e+z^2)]^{-1}.
\]
As $|z|\to\infty$,
\[
 f(z)=O\!\left(|z|^{-1/2}(\log|z|)^{-1}\right),\qquad
 f'(z)=O\!\left(|z|^{-3/2}(\log|z|)^{-1}\right),
\]
and
\[
 f''(z)=O\!\left(|z|^{-5/2}(\log|z|)^{-1}\right).
\]
Hence $f\in H^2(\R)\cap L^\infty(\R)$ and $f\in L^2(\R)$, whereas $f\notin L^q(\R)$ for every $q<2$: for large $|z|$,
\[
 |f(z)|^q\gtrsim |z|^{-q/2}(\log|z|)^{-q},
\]
and $q/2<1$.  Define the pure-swirl field
\[
 u_0=(-x_2\chi(r)f(z),\ x_1\chi(r)f(z),\ 0),
\]
so that $u_0^\theta=r\chi(r)f(z)$.  The field is smooth, divergence free, belongs to $H^2(\R^3)$, and
\[
 \Gamma_0=r^2\chi(r)f(z)\in L^\infty.
\]
For every $q\ge3/2$,
\[
 |Y_q|^q\,\dd x
 =2\pi r^{2q-2}|\chi(r)|^q|f(z)|^q\,\dd r\,\dd z.
\]
The radial integral is finite and strictly positive because $2q-2>-1$ and $\chi$ is compactly supported.  Hence $Y_q\in L^q(\R^3)$ if and only if $f\in L^q(\R)$.  The stated sharpness follows.
\end{proof}

\begin{remark}\label{rem:moment-tilt}
For $q\ge3$ the critical density has the exact representation
\begin{equation}\label{eq:moment-tilt}
 |Y_q|^q=|u^\theta|^3|\Gamma|^{q-3},
 \qquad
 \overline U_q
 =\frac{\int U|u^\theta|^3|\Gamma|^{q-3}\,\dd\mu_3}
 {\int |u^\theta|^3|\Gamma|^{q-3}\,\dd\mu_3}.
\end{equation}
Thus $q=3$ weights the strain by cubic swirl density, while larger $q$ tilt the same signed average toward regions carrying larger circulation.  No monotonicity in $q$ is asserted.
\end{remark}

\begin{corollary}\label{cor:q3-universal}
Let $u$ be a smooth axisymmetric $H^2$ solution on $[0,T)$; no bounded-circulation hypothesis is needed in this corollary.  Define
\[
 \overline U_3(t)
 =\frac{\int U|u^\theta|^3\,\dd\mu_3}{\int|u^\theta|^3\,\dd\mu_3}.
\]
If for some $t_0<T$
\[
 \sup_{t_0<t<T}\int_{t_0}^t-\overline U_3(s)\,\dd s<\infty,
\]
then $u$ extends smoothly through $T$.
\end{corollary}

\begin{proof}
Here $q=3$, so $d_q=0$ and $Y_q=u^\theta$.  The critical mass is finite by $H^1$ interpolation.  The proof of \cref{thm:two-parameter-action} gives
$u^\theta\in L_t^4L_x^6$.  Since $d=0$ lies in Theorem~1.1 of Chen--Fang--Zhang \cite{ChenFangZhang2017}, no circulation interpolation is required.
\end{proof}

\begin{corollary}\label{cor:critical-form-bound}
Let $u$ be a smooth axisymmetric $H^2$ solution on $[0,T)$ and fix $3/2\le q<\infty$.  Assume $Y_q(t_0)\in L^q$ for some $0<t_0<T$; when $3/2\le q<3$, also assume bounded initial circulation $\Gamma_0\in L^\infty$.  Let $K_q=|Y_q|^{q/2}$.  If, for almost every $t\in(t_0,T)$,
\begin{equation}\label{eq:critical-form-bound}
 \int U_-K_q^2\,\dd\mu_3
 \le\Lambda\int|\nabla K_q|^2\,\dd\mu_3,
 \qquad
 \Lambda<\frac{4\nu(q-1)}{3q},
\end{equation}
then $u$ extends smoothly through $T$.
\end{corollary}

\begin{proof}
In \eqref{eq:critical-balance}, discard the nonnegative Hardy term (or the nonnegative axis trace at $q=3/2$) and estimate
\[
 -3\int U K_q^2\,\dd\mu_3
 \le3\int U_-K_q^2\,\dd\mu_3
 \le3\Lambda\int|\nabla K_q|^2\,\dd\mu_3.
\]
The strict inequality in \eqref{eq:critical-form-bound} leaves positive gradient dissipation, so $Z_q$ remains bounded and $K_q\in L_t^2\dot H_x^1$.  The mixed-norm continuation argument of \cref{thm:two-parameter-action} applies.
\end{proof}

The left endpoint $q=3/2$ requires genuine strong critical mass.  The following example from the $F$ endpoint remains useful.

\begin{proposition}\label{prop:weak-endpoint}
There exist smooth divergence-free axisymmetric data $u_0\in H^2(\R^3)$ with finite energy and bounded circulation $\Gamma_0=ru_0^\theta$ such that
\[
 F_0=\frac{u_0^\theta}{r}
 \in L^{3/2,\infty}(\R^3)\setminus L^{3/2}(\R^3).
\]
\end{proposition}

\begin{proof}
Choose $\chi\in C_c^\infty([0,\infty))$ with $\chi=1$ on $[0,1]$ and write $\chi(r)=\chi_0(r^2)$ near the axis.  Let
\[
 f(z)=(1+z^2)^{-1/3},
 \qquad
 F_0(r,z)=\chi(r)f(z),
 \qquad
 u_0=(-x_2F_0,x_1F_0,0).
\]
Then $u_0=u_0^\theta e_\theta$ with $u_0^\theta=rF_0$.  It is smooth, divergence free, belongs to $H^2(\R^3)$, has finite energy, and has bounded circulation $\Gamma_0=r^2\chi(r)f(z)$.  For small $\lambda>0$, the one-dimensional measure of $\{f>\lambda\}$ is comparable to $\lambda^{-3/2}$, and the radial factor is nonzero on a fixed disk.  Hence
\[
 |\{|F_0|>\lambda\}|\simeq\lambda^{-3/2},
\]
so $F_0\in L^{3/2,\infty}$.  But
\[
 \int_{\R^3}|F_0|^{3/2}\,\dd x
 =C_\chi\int_\R(1+z^2)^{-1/2}\,\dd z=\infty.
\]
\end{proof}

\begin{remark}\label{rem:lorentz-scope}
Mixed and anisotropic Lorentz-space criteria for axisymmetric Navier--Stokes are available in the literature \cite{WangHuangWeiYu2024}.  Proposition~\ref{prop:weak-endpoint} does not rule out a weak-$L^{3/2}$ continuation theorem of a different type.  It shows that the left endpoint of the present action family cannot be obtained by the formal substitution $L^{3/2}\rightsquigarrow L^{3/2,\infty}$, because the denominator defining $\overline U_{3/2}$ may be infinite.
\end{remark}

\begin{corollary}\label{cor:continuum-necessary-action}
Let $T$ be a first singular time for a smooth axisymmetric $H^2$ solution with bounded circulation.  For every sufficiently late $t_0<T$ and every finite $q\ge2$,
\begin{equation}\label{eq:continuum-action-diverge}
 \sup_{t_0<t<T}\int_{t_0}^t-\overline U_q(s)\,\dd s=\infty.
\end{equation}
If $Y_q(t_0)\in L^q$ is additionally known for some $3/2\le q<2$, then \eqref{eq:continuum-action-diverge} holds at that exponent as well.
\end{corollary}

\begin{proof}
For $q\ge2$, \cref{prop:automatic-critical-mass} supplies the critical mass and \cref{cor:critical-family} supplies continuation whenever the action is bounded.  The statement is the contrapositive.  The same argument applies in $3/2\le q<2$ whenever the corresponding mass is finite.
\end{proof}

\section{Moving-boundary heat estimates}

We next isolate the scalar estimate behind the circulation criterion.  Let $g\ge0$ and consider
\begin{equation}\label{eq:scalar}
 \partial_t q=\nu q_{rr}+g(t)q_r,
 \qquad r>0,
 \qquad q(0,t)=0,
 \qquad 0\le q\le M.
\end{equation}
Set
\[
 A(t)=\int_0^t g(s)\,\dd s,
 \qquad
 D_\sigma(h)=A(\sigma)-A(\sigma-h).
\]
Fix $\sigma>0$ and a reference window $0<h_0<\sigma$.  Write $\tau=\sigma-t$ and
\[
 s=\log\frac{h_0}{\tau},\qquad \tau=h_0e^{-s}.
\]
For $s\ge0$ define
\begin{equation}\label{eq:a-def}
 a_\sigma(s)=\frac{D_\sigma(h_0e^{-s})}{\sqrt{\nu h_0e^{-s}}}.
\end{equation}

The following weighted Hardy inequality is the quantitative input.

\begin{lemma}\label{lem:gaussHardy}
Let $a\ge0$ and $\rho(y)=e^{-y^2/4}$.  If $f(-a)=0$, then
\begin{equation}\label{eq:gaussHardy}
 \int_{-a}^{\infty}f(y)^2\rho(y)\,\dd y
 \le C\frac{e^{a^2/4}}{1+a}
 \int_{-a}^{\infty}|f'(y)|^2\rho(y)\,\dd y,
\end{equation}
where $C$ is universal.
\end{lemma}

\begin{proof}
The one-dimensional weighted Hardy criterion of Muckenhoupt \cite{Muckenhoupt1972} bounds the optimal constant in \eqref{eq:gaussHardy} by $4B_a$, where
\[
 B_a=\sup_{x>-a}
 \left(\int_x^\infty\rho(y)\,\dd y\right)
 \left(\int_{-a}^x\rho(y)^{-1}\,\dd y\right).
\]
For completeness we estimate $B_a$ directly.  If $-a<x\le0$, then
\[
 \int_x^\infty\rho\le C,
 \qquad
 \int_{-a}^x e^{y^2/4}\,\dd y
 \le\int_{-a}^0e^{y^2/4}\,\dd y
 \le C\frac{e^{a^2/4}}{1+a}.
\]
For $a\ge1$, the part $[0,a/2]$ is at most $(a/2)e^{a^2/16}$, which is bounded by $Ce^{a^2/4}/a$.  On $[a/2,a]$ we use
$a^2-y^2=(a-y)(a+y)\ge a(a-y)$ to obtain
\[
 \int_{a/2}^{a}e^{y^2/4}\,\dd y
 \le e^{a^2/4}\int_0^\infty e^{-as/4}\,\dd s
 =\frac4a e^{a^2/4}.
\]
The case $0\le a\le1$ is trivial after enlarging the universal constant.

If $x\ge0$, the Gaussian tail bound gives
\[
 \int_x^\infty e^{-y^2/4}\,\dd y
 \le C\frac{e^{-x^2/4}}{1+x},
\]
while
\[
 \int_{-a}^x e^{y^2/4}\,\dd y
 \le C\left(
 \frac{e^{a^2/4}}{1+a}+\frac{e^{x^2/4}}{1+x}
 \right).
\]
Multiplication yields
\[
 \left(\int_x^\infty\rho\right)
 \left(\int_{-a}^x\rho^{-1}\right)
 \le C\left(
 \frac{e^{a^2/4}}{1+a}+1
 \right)
 \le C\frac{e^{a^2/4}}{1+a},
\]
after enlarging the universal constant.  Thus $B_a\le C e^{a^2/4}/(1+a)$, which proves the lemma.
\end{proof}

\begin{theorem}\label{thm:scalar-general}
Let $q$ solve \eqref{eq:scalar}.  There exist universal constants $c,C>0$ such that, for every fixed terminal time $\sigma$ and all sufficiently small $x>0$, let
\[
 S_x=\log\frac{\nu h_0}{4x^2}.
\]
Then
\begin{equation}\label{eq:scalar-general-est}
 q(x,\sigma)
 \le CM\exp\left
 \{-c\int_{0}^{S_x}
 (1+a_\sigma(s))e^{-a_\sigma(s)^2/4}\,\dd s\right
 \}.
\end{equation}
Changing the fixed reference window $h_0$ changes only the constants and an additive bounded part of the exponent.
\end{theorem}

\begin{proof}
Set
\[
 z=r+A(t),
 \qquad
 v(z,t)=q(z-A(t),t).
\]
Since $A'=g$, direct differentiation gives
\[
 v_t=\nu v_{zz}
 \qquad\text{for }z>A(t),
 \qquad
 v(A(t),t)=0.
\]
Fix $\sigma$, let $\tau=\sigma-t$, $s=\log(h_0/\tau)$, and introduce
\[
 y=\frac{z-A(\sigma)}{\sqrt{\nu\tau}},
 \qquad
 w(y,s)=v(z,t).
\]
Then
\begin{equation}\label{eq:OU}
 w_s=w_{yy}-\frac y2w_y
 =\rho^{-1}(\rho w_y)_y,
 \qquad
 y>-a_\sigma(s),
 \qquad
 w(-a_\sigma(s),s)=0,
\end{equation}
with $\rho=e^{-y^2/4}$.

Define
\[
 E(s)=\int_{-a_\sigma(s)}^\infty w(y,s)^2\rho(y)\,\dd y.
\]
The moving-boundary contribution in $E'$ vanishes because $w=0$ at the boundary.  Integrating \eqref{eq:OU} by parts therefore gives
\[
 E'(s)=-2\int_{-a_\sigma(s)}^\infty |w_y|^2\rho\,\dd y.
\]
By \cref{lem:gaussHardy},
\begin{equation}\label{eq:Edecay}
 E'(s)
 \le-c(1+a_\sigma(s))e^{-a_\sigma(s)^2/4}E(s).
\end{equation}
Since $0\le w\le M$, $E(0)\le CM^2$.  Integrating \eqref{eq:Edecay} gives
\begin{equation}\label{eq:Eint}
 E(S)\le CM^2\exp\left\{-c\int_{0}^{S}
 (1+a_\sigma(s))e^{-a_\sigma(s)^2/4}\,\dd s\right\}.
\end{equation}

It remains to convert the Gaussian energy to a terminal pointwise bound.  Fix $x>0$ and set $t_0=\sigma-4x^2/\nu$; equivalently $\tau_0=4x^2/\nu$, so that $\sqrt{\nu\tau_0}=2x$.  Extend $v(\cdot,t_0)$ by zero below $A(t_0)$.  Since the moving Dirichlet boundary can only remove heat, the maximum principle gives domination by the free heat semigroup on $[t_0,\sigma]$.  Hence, at $z=A(\sigma)+x$,
\begin{align*}
 q(x,\sigma)
 &=v(A(\sigma)+x,\sigma)\\
 &\le \frac1{\sqrt{4\pi\nu(\sigma-t_0)}}
 \int_{A(t_0)}^\infty
 e^{-\frac{(A(\sigma)+x-z')^2}{4\nu(\sigma-t_0)}}
 v(z',t_0)\,\dd z'.
\end{align*}
With $z'=A(\sigma)+2xy$, this becomes
\[
 q(x,\sigma)
 \le C\int_{-a_\sigma(s_0)}^\infty
 e^{-(1-2y)^2/16}w(y,s_0)\,\dd y,
 \qquad
 s_0=\log(\nu h_0/(4x^2)).
\]
Cauchy--Schwarz with weight $\rho$ yields
\begin{align*}
 q(x,\sigma)
 &\le CE(s_0)^{1/2}
 \left(\int_\R e^{-(1-2y)^2/8}e^{y^2/4}\,\dd y\right)^{1/2}\\
 &\le CE(s_0)^{1/2},
\end{align*}
because
\[
 -\frac{(1-2y)^2}{8}+\frac{y^2}{4}
 =\frac18-\frac{(y-1)^2}{4}.
\]
Combining this with \eqref{eq:Eint} and $s_0=2|\log x|+O_{\nu,h_0}(1)$ proves \eqref{eq:scalar-general-est}.
\end{proof}

The classical Petrovskii constant now follows directly.

\begin{corollary}\label{cor:scalar-K}
Suppose that for some $0\le K\le2$,
\begin{equation}\label{eq:scalar-K-cond}
 D_\sigma(h)
 \le K\sqrt{\nu h\ell_{h_0}(h)}
\end{equation}
for all sufficiently small $h$.  Then, as $x\downarrow0$,
\begin{equation}\label{eq:scalar-K-decay}
 q(x,\sigma)\le
 \begin{cases}
 CM\exp\{-c|\log x|^{1-K^2/4}\},&0\le K<2,\\[1mm]
 CM\exp\{-c(\log|\log x|)^{3/2}\},&K=2.
 \end{cases}
\end{equation}
\end{corollary}

\begin{proof}
The hypothesis gives
\[
 a_\sigma(s)\le K\sqrt{\log(e+s)}.
\]
For $K<2$, \eqref{eq:Edecay} may be weakened to
\[
 E'(s)\le-c(e+s)^{-K^2/4}E(s),
\]
and integration gives the first line of \eqref{eq:scalar-K-decay}.  If $K=2$, retain the factor $1+a$ in \eqref{eq:Edecay}.  Put $L=\log(e+s)$.  On $0\le a\le2\sqrt L$, the function $(1+a)e^{-a^2/4}$ is bounded below either by a fixed positive constant (when $a$ stays in a bounded interval) or, on its decreasing branch, by its endpoint value.  Hence for large $s$,
\[
 (1+a_\sigma(s))e^{-a_\sigma(s)^2/4}
 \ge c\frac{\sqrt{\log(e+s)}}{e+s},
\]
so
\[
 E(s)\le CM^2\exp\{-c(\log(e+s))^{3/2}\}.
\]
The pointwise step in \cref{thm:scalar-general} gives the second line.
\end{proof}

\begin{remark}
The inclusive endpoint $K=2$ is the classical Petrovskii threshold for the corresponding moving-boundary heat geometry.  The full Kolmogorov--Petrovskii regularity test contains finer lower-order information at the endpoint; see \cite{Abdulla2005}.  We use only the quantitative sufficient form needed for the circulation comparison.
\end{remark}

\begin{proposition}\label{prop:terminal-no-go}
A displacement condition imposed only on windows ending at one final time cannot replace the sliding hypothesis.  More precisely, after normalizing $\nu=1$ and $T=1$, for every $K>0$ there exist a nonnegative $g\in C^\infty([0,1))$ and a bounded nonnegative solution of
\[
 q_t=q_{rr}+g(t)q_r,\qquad q(0,t)=0,
\]
with smooth bounded nondecreasing initial data, such that
\begin{equation}\label{eq:terminal-only-wall}
 \int_{1-h}^{1}g(t)\,\dd t
 \le K\sqrt{h\log\log(e^e/h)}
\end{equation}
for every sufficiently small $h$, but there are $t_n\uparrow1$ and $r_n\downarrow0$ for which
\begin{equation}\label{eq:terminal-no-modulus}
 q(r_n,t_n)|\log r_n|^{3/2}\longrightarrow\infty.
\end{equation}
In particular, the terminal-only bound \eqref{eq:terminal-only-wall} does not imply a uniform Wei-type logarithmic modulus on preterminal times.
\end{proposition}

\begin{proof}
Set
\[
 W(h)=\sqrt{h\log\log(e^e/h)}.
\]
Choose $n_0$ large and
\[
 \tau_n=e^{-2^n},\qquad t_n=1-\tau_n,\qquad
 \delta_n=\frac K8 W(\tau_n),\qquad n\ge n_0.
\]
Since $\tau_{n+1}=\tau_n^2$, after increasing $n_0$ we may assume
\begin{equation}\label{eq:deltatail}
 \delta_{n+1}\le\frac14\delta_n,
 \qquad
 \sum_{j\ge n}\delta_j\le\frac43\delta_n.
\end{equation}
Let
\[
 \varepsilon_n=\exp(-\delta_n^{-4}).
\]
Then $\varepsilon_n$ is much smaller than both $\tau_{n+1}$ and the separation $t_{n+1}-t_n$ for large $n$.  Fix $\eta\in C_c^\infty(0,1)$, $\eta\ge0$, $\int_0^1\eta=1$, and put
\[
 g_n(t)=\frac{\delta_n}{\varepsilon_n}
 \eta\!\left(\frac{t-(t_n-\varepsilon_n)}{\varepsilon_n}\right),
 \qquad
 g=\sum_{n\ge n_0}g_n.
\]
The supports are disjoint and locally finite in $[0,1)$, so $g$ is smooth there and each pulse has mass $\int g_n=\delta_n$.

The function $W$ is increasing for all sufficiently small arguments.  Given small $h$, choose $n$ so that $\tau_n\le h<\tau_{n-1}$.  The terminal interval $(1-h,1)$ meets only a part of the $n$th pulse and the full later pulses.  Hence, by \eqref{eq:deltatail},
\[
 \int_{1-h}^1g
 \le\sum_{j\ge n}\delta_j
 \le\frac43\frac K8W(\tau_n)
 \le K W(h),
\]
which proves \eqref{eq:terminal-only-wall}.

It remains to show that each short pulse creates a much thinner preterminal layer than the terminal accounting sees.  Choose smooth bounded increasing initial data
\[
 q_0(r)=\frac{2}{\sqrt\pi}\int_0^r e^{-y^2}\,\dd y.
\]
For each $h>0$, the difference $q(r+h,t)-q(r,t)$ solves the same linear equation on $r>0$ with nonnegative initial data and nonnegative boundary value at $r=0$; the maximum principle therefore gives $q(r+h,t)\ge q(r,t)$.  Hence $q_r\ge0$.  Therefore
\[
 q_t-q_{rr}=gq_r\ge0,
\]
and comparison with the drift-free Dirichlet heat solution $q^{(0)}$ gives $q\ge q^{(0)}$.  By the explicit Dirichlet heat kernel, $\partial_rq^{(0)}(0,t)>0$ for every $t>0$; continuity on the compact interval $[1/2,1]$ yields constants $c_0,r_0>0$ such that
\begin{equation}\label{eq:baseline-linear}
 q(r,t)\ge q^{(0)}(r,t)\ge c_0r,
 \qquad 0<r<r_0,\quad \frac12<t<1.
\end{equation}

Consider the $n$th pulse and set $s_n=t_n-\varepsilon_n$ and
\[
 r_n=4\sqrt{\varepsilon_n}.
\]
During $[s_n,t_n]$ the killed diffusion representation for the time-dependent generator $\partial_{rr}+g(t)\partial_r$ is
\[
 q(r_n,t_n)=
 \mathbb E\left[q(X_{\varepsilon_n},s_n)
 \mathbf 1_{\{\tau_0>\varepsilon_n\}}\right],
\]
where
\[
 X_s=r_n+\int_0^s g(t_n-\xi)\,\dd\xi+\sqrt2 B_s
\]
and $\tau_0$ is the first hitting time of $0$.  The coefficient is read backward from the observation time $t_n$ because this is the Feynman--Kac representation of the forward initial-value problem from $s_n$ to $t_n$.  The deterministic drift is nevertheless nonnegative and has total increment
$\int_0^{\varepsilon_n}g(t_n-\xi)\,\dd\xi=\delta_n$.  Since $\delta_n/\sqrt{\varepsilon_n}\to\infty$, consider the event
\[
 \inf_{0\le s\le\varepsilon_n}\sqrt2B_s\ge-2\sqrt{\varepsilon_n},
 \qquad
 |\sqrt2B_{\varepsilon_n}|\le\frac14\delta_n.
\]
By Brownian scaling and the reflection principle, the first condition has a fixed positive probability, while the probability of the second tends to one.  Their intersection therefore has probability at least a constant $p_0>0$ for all large $n$.  On this event the path remains positive because $r_n=4\sqrt{\varepsilon_n}$ and the deterministic drift is nonnegative.  Moreover, since $4\sqrt{\varepsilon_n}\le\delta_n/4$ for large $n$ and the total deterministic increment equals $\delta_n$,
\[
 \frac12\delta_n\le X_{\varepsilon_n}\le2\delta_n.
\]
For large $n$ this interval lies in $(0,r_0)$.  Applying \eqref{eq:baseline-linear} at time $s_n$ therefore gives
\begin{equation}\label{eq:pulse-lower}
 q(r_n,t_n)\ge c_1\delta_n
\end{equation}
with $c_1>0$ independent of $n$.

Finally,
\[
 |\log r_n|=\frac12\delta_n^{-4}+O(1),
\]
so \eqref{eq:pulse-lower} implies
\[
 q(r_n,t_n)|\log r_n|^{3/2}
 \ge c\delta_n\,\delta_n^{-6}
 =c\delta_n^{-5}\longrightarrow\infty.
\]
This proves \eqref{eq:terminal-no-modulus}.  The normalization $\nu=1$, $T=1$ is removed by the usual parabolic scaling.
\end{proof}

\section{Sliding Petrovskii regularity}

We now return to the circulation equation \eqref{eq:Gamma}.

\begin{lemma}\label{lem:comparison}
Let $u$ be smooth and axisymmetric on $[0,T)$.  Set $m(t)=\norm{(u^r)_-(t)}_\infty$.  There exists a bounded nonnegative nondecreasing function $q=q(r,t)$ satisfying
\begin{equation}\label{eq:qcomp}
 q_t=\nu q_{rr}+m(t)q_r,
 \qquad q(0,t)=0,
\end{equation}
and
\begin{equation}\label{eq:GammaComp}
 |\Gam(r,z,t)|\le q(r,t)
\end{equation}
for all $r,z,t$.  Moreover $\norm{q}_\infty\le\norm{\Gam_0}_\infty$ after choosing the initial envelope optimally.
\end{lemma}

\begin{proof}
Because $u_0\in H^2(\R^3)$, Sobolev embedding gives $u_0^\theta\in L^\infty$.  Hence
$|\Gam(r,z,0)|\le r\|u_0^\theta\|_\infty$ near the axis, uniformly in $z$.  Define the monotone envelope
\[
 q_0^*(r)=\sup_{0\le\rho\le r,\ z\in\R}|\Gam(\rho,z,0)|.
\]
Then $q_0^*(0)=0$, $q_0^*$ is nonnegative and nondecreasing, and
$q_0^*\le\|\Gam_0\|_\infty$.  Approximate $q_0^*$ monotonically from above by smooth nondecreasing envelopes and pass to the limit in the one-dimensional parabolic problem.  The resulting bounded solution $q$ of \eqref{eq:qcomp} is nonnegative and nondecreasing in $r$, and the maximum principle gives
$\|q\|_\infty\le\|\Gam_0\|_\infty$.

Define the axisymmetric circulation operator
\[
 \mathcal L f
 =f_t-\nu\left(f_{rr}-\frac1r f_r+f_{zz}\right)
 +u^rf_r+u^zf_z.
\]
Equation \eqref{eq:Gamma} says $\mathcal L\Gam=0$.  Since $q$ is independent of $z$ and satisfies \eqref{eq:qcomp},
\[
 \mathcal L q
 =\left(m(t)+u^r+\frac\nu r\right)q_r\ge0,
\]
because $m(t)+u^r\ge0$ and $q_r\ge0$.  Thus $q$ is a supersolution.  Apply the maximum principle first on truncated cylinders
$\{\varepsilon<r<R,\ |z|<R\}$, using the axis compatibility $\Gam(0,z,t)=q(0,t)=0$ and the decay of the strong solution to control the artificial outer boundary, and then let $\varepsilon\downarrow0$ and $R\uparrow\infty$.  Applying the argument to both $\Gam-q$ and $-\Gam-q$ yields \eqref{eq:GammaComp}.
\end{proof}

\begin{corollary}\label{thm:petro-integral}
Let $u$ be a smooth axisymmetric solution of \eqref{eq:NS} on $[0,T)$ with $u_0\in H^2(\R^3)$ smooth, finite energy, and $\Gamma_0\in L^\infty$.  Fix $0<h_0<T/2$ sufficiently small and, for $\sigma\in(T-h_0,T)$, define
\[
 a_\sigma(s)=
 \frac{D_\sigma(h_0e^{-s})}{\sqrt{\nu h_0e^{-s}}},
 \qquad
 \mathfrak P_\sigma(S)=
 \int_0^S(1+a_\sigma(s))e^{-a_\sigma(s)^2/4}\,\dd s.
\]
If
\begin{equation}\label{eq:petro-integral-NS}
 \lim_{S\to\infty}
 \inf_{\sigma\in(T-h_0,T)}
 \frac{\mathfrak P_\sigma(S)}{\log(e+S)}=\infty,
\end{equation}
then $u$ extends smoothly through $T$.
\end{corollary}

\begin{proof}
Apply \cref{lem:comparison} and then \cref{thm:scalar-general} with terminal time equal to each $\sigma\in(T-h_0,T)$.  The $L^\infty$ norm of the comparison function is uniformly bounded by $\|\Gamma_0\|_\infty$.  If $r$ is small and
\[
 S_r=\log\frac{\nu h_0}{4r^2},
\]
then
\[
 |\Gamma(r,z,\sigma)|
 \le C\|\Gamma_0\|_\infty e^{-c\mathfrak P_\sigma(S_r)}.
\]
Condition \eqref{eq:petro-integral-NS} implies that the right side is, uniformly in late $\sigma$, smaller than $|\log r|^{-N}$ for every fixed $N>0$ once $r$ is sufficiently small.  Choose any restart time $t_1\in(T-h_0,T)$.  Taking $N=3/2$ and shrinking the axis radius if necessary gives Wei's circulation hypothesis uniformly on $[t_1,T)$.  The shifted solution has $H^2$ data at $t_1$ and bounded circulation, so Wei's criterion \cite{Wei2016} applies on the shifted interval and yields continuation through $T$.
\end{proof}

\begin{theorem}\label{thm:petro-NS}
Let $u$ be a smooth axisymmetric solution of \eqref{eq:NS} on $[0,T)$ with $u_0\in H^2(\R^3)$ smooth, finite energy, and $\Gam_0\in L^\infty$.  Assume that for some $0<h_0<T/2$,
\begin{equation}\label{eq:petroNS}
 \int_{\sigma-h}^{\sigma}\norm{(u^r)_-(t)}_\infty\,\dd t
 \le2\sqrt{\nu h\ell_{h_0}(h)}
\end{equation}
for every $\sigma\in(T-h_0,T)$ and every $0<h<\min\{h_0,\sigma\}$.  Then $u$ extends smoothly through $T$.
\end{theorem}

\begin{proof}
Apply \cref{lem:comparison} and then \cref{cor:scalar-K} with $K=2$ and terminal time equal to each $\sigma\in(T-h_0,T)$.  The constants in the scalar estimate are uniform because the sliding hypothesis and $\norm{q}_\infty\le\norm{\Gam_0}_\infty$ are uniform.  Thus there exist $c,C>0$ such that, for all sufficiently small $r$ and all $\sigma$ sufficiently close to $T$,
\begin{equation}\label{eq:GammaEndpointMod}
 |\Gam(r,z,\sigma)|
 \le C\norm{\Gam_0}_\infty
 \exp\{-c(\log|\log r|)^{3/2}\}.
\end{equation}
For every fixed $N>0$, the right side of \eqref{eq:GammaEndpointMod} is eventually smaller than $|\log r|^{-N}$ as $r\downarrow0$.  Choose a restart time $t_1\in(T-h_0,T)$.  After shrinking the axis cylinder if necessary,
\[
 |\Gam(r,z,t)|\le |\log r|^{-3/2}
\]
uniformly for $t\in[t_1,T)$.  The solution at $t_1$ is again an $H^2$ axisymmetric strong datum and its circulation is bounded by the maximum principle.  Wei's regularity criterion \cite{Wei2016}, applied to this shifted solution, excludes singularity at $T$ and gives continuation.
\end{proof}

\begin{corollary}\label{cor:petro-secondary}
Let the hypotheses of \cref{thm:petro-NS} hold except for \eqref{eq:petroNS}.  Suppose that there exist constants $\beta<1/2$, $C_0\in\R$, and $0<h_0<T/2$ such that, for every $\sigma\in(T-h_0,T)$ and every sufficiently small $0<h<h_0$,
\begin{equation}\label{eq:petro-secondary}
 \frac{D_\sigma(h)^2}{4\nu h}
 \le
 \ell_{h_0}(h)
 +\beta\log\!\bigl(e+\ell_{h_0}(h)\bigr)+C_0.
\end{equation}
Then $u$ extends smoothly through $T$.
\end{corollary}

\begin{proof}
Put $L(s)=\log(e+s)$.  Since $\ell_{h_0}(h_0e^{-s})=L(s)$, \eqref{eq:petro-secondary} gives, after changing the additive constant,
\[
 \frac{a_\sigma(s)^2}{4}
 \le L(s)+\beta\log(e+L(s))+C.
\]
For large $s$, either $a_\sigma(s)$ stays in a bounded interval, in which case $(1+a_\sigma)e^{-a_\sigma^2/4}$ is bounded below by a positive constant, or it lies on the decreasing branch of $a\mapsto(1+a)e^{-a^2/4}$.  In the latter case,
\[
 (1+a_\sigma(s))e^{-a_\sigma(s)^2/4}
 \ge
 c\frac{L(s)^{1/2-\beta}}{e+s}.
\]
Consequently, uniformly in $\sigma$,
\[
 \mathfrak P_\sigma(S)
 \ge c[\log(e+S)]^{3/2-\beta}-C.
\]
Because $\beta<1/2$, this quantity divided by $\log(e+S)$ tends to infinity.  The general integral criterion \cref{thm:petro-integral} therefore gives continuation.
\end{proof}

\begin{remark}\label{rem:petro-leading}
The coefficient $2$ in \eqref{eq:petroNS} is the sharp \emph{leading} Petrovskii coefficient.  Corollary~\ref{cor:petro-secondary} shows that it is not an absolute lower-order wall: a positive $\log\log\log$ correction is still allowed as long as its coefficient in \eqref{eq:petro-secondary} is strictly below $1/2$.  This is consistent with the classical Kolmogorov--Petrovskii theory, where lower iterated logarithms resolve the critical leading coefficient \cite{Abdulla2005}.  The threshold $\beta<1/2$ here is the point at which our quantitative circulation decay remains stronger than every negative power of $|\log r|$, as required to invoke Wei's criterion; it is not claimed to be the full scalar boundary-regularity threshold.
\end{remark}

\begin{corollary}\label{cor:partial-typeI}
The criterion in \cref{thm:petro-NS} contains every partial type-I bound
\[
 \norm{(u^r)_-(t)}_\infty\le\frac{C}{\sqrt{T-t}}.
\]
More generally, it permits terminal inward velocities whose sliding averages are larger than the type-I scale by a factor of order $\sqrt{\log\log(1/h)}$.
\end{corollary}

\begin{proof}
For a partial type-I bound and any window ending at $\sigma<T$,
\[
 \int_{\sigma-h}^{\sigma}\frac{C}{\sqrt{T-t}}\,\dd t
 \le2C\sqrt h.
\]
Since $\sqrt{\ell_{h_0}(h)}\to\infty$ as $h\downarrow0$, the right side is eventually dominated by the Petrovskii wall in \eqref{eq:petroNS}.  Away from the terminal time the velocity is bounded by smoothness, so the same conclusion holds uniformly on all sufficiently small sliding windows.
\end{proof}

\section{Source-to-compression response}

The next identity isolates the first infinitesimal leg in the feedback chain $F\to G\to U$.  It is independent of the continuation criteria above.

Let
\[
 h=F^2,
 \qquad
 \mathcal A=-\lapfive.
\]
At a fixed time, freeze $F$ and let the $G$-field receive only the source increment
\begin{equation}\label{eq:dGsrc}
 \dot G_{\rm src}=\partial_z h.
\end{equation}
Under the Hodge map \eqref{eq:Hodge}--\eqref{eq:U-Hodge}, write
\begin{equation}\label{eq:phi-src}
 \phi_{\rm src}=\mathcal A^{-1}\partial_z h,
 \qquad
 \dot U_{\rm src}=-\partial_z\phi_{\rm src}.
\end{equation}

\begin{theorem}\label{thm:weighted-antisource}
For $\alpha\ge1$, define the power-weighted compressive functional
\[
 \mathcal K_{2,\alpha}(F,U)
 =-\int UF^2r^\alpha\,\dd r\,\dd z.
\]
For every smooth rapidly decaying axis-compatible profile $F$, the pure source increment \eqref{eq:dGsrc}--\eqref{eq:phi-src} satisfies
\begin{equation}\label{eq:weighted-source-response}
 \int \dot U_{\rm src}F^2r^\alpha\,\dd r\,\dd z
 =\mathcal Q_\alpha[\phi_{\rm src}],
\end{equation}
where
\begin{equation}\label{eq:Qalpha-source}
 \mathcal Q_1[\phi]
 =\int|\nabla\phi|^2r\,\dd r\,\dd z
 +\int_\R|\phi(0,z)|^2\,\dd z,
\end{equation}
and, for $\alpha>1$,
\begin{equation}\label{eq:Qalpha-source2}
 \mathcal Q_\alpha[\phi]
 =\int|\nabla\phi|^2r^\alpha\,\dd r\,\dd z
 +\frac{(3-\alpha)(\alpha-1)}2
  \int \phi^2r^{\alpha-2}\,\dd r\,\dd z.
\end{equation}
Consequently
\[
 \mathcal Q_\alpha[\phi]\ge0
 \qquad\text{for }1\le\alpha\le5.
\]
More quantitatively, for $1<\alpha\le3$,
\[
 \mathcal Q_\alpha[\phi]\ge
 \int|\nabla\phi|^2r^\alpha\,\dd r\,\dd z,
\]
while for $3<\alpha<5$,
\begin{equation}\label{eq:source-gap-alpha}
 \mathcal Q_\alpha[\phi]\ge
 \frac{5-\alpha}{\alpha-1}
 \int|\nabla\phi|^2r^\alpha\,\dd r\,\dd z.
\end{equation}
At $\alpha=5$ the form is nonnegative but has no positive uniform full-gradient gap.  Hence, throughout $1\le\alpha\le5$,
\[
 \left.\frac{\dd}{\dd\varepsilon}\right|_{\varepsilon=0}
 \mathcal K_{2,\alpha}(F,U+\varepsilon\dot U_{\rm src})\le0.
\]
\end{theorem}

\begin{proof}
Put $h=F^2$ and $\phi=\phi_{\rm src}$.  Integration in $z$ gives
\[
 \int \dot U_{\rm src}h r^\alpha\,\dd r\,\dd z
 =-\int\phi_z h r^\alpha\,\dd r\,\dd z
 =\int\phi h_z r^\alpha\,\dd r\,\dd z
 =\int\phi\mathcal A\phi\,r^\alpha\,\dd r\,\dd z.
\]
For $\alpha>1$, radial integration by parts yields
\[
 \int\phi\mathcal A\phi\,r^\alpha
 =\int|\nabla\phi|^2r^\alpha
 +\frac{(3-\alpha)(\alpha-1)}2\int\phi^2r^{\alpha-2},
\]
with all integrals over the meridional half-plane.  At $\alpha=1$ the axis boundary survives and gives \eqref{eq:Qalpha-source}.  For $1<\alpha\le3$ the zero-order coefficient is nonnegative.  For $3<\alpha\le5$, the sharp radial Hardy inequality
\[
 \int|\phi_r|^2r^\alpha\,\dd r\,\dd z
 \ge\frac{(\alpha-1)^2}{4}
 \int\phi^2r^{\alpha-2}\,\dd r\,\dd z
\]
gives \eqref{eq:source-gap-alpha} for $\alpha<5$ and nonnegativity at $\alpha=5$.  The absence of a positive full-gradient gap at $\alpha=5$ follows from the sharpness of Hardy's inequality, equivalently from the $G$-diffusion endpoint in \cite{ShahmurovHardy2026}.  Differentiating $\mathcal K_{2,\alpha}$ with $F$ frozen proves the final assertion.
\end{proof}

\begin{corollary}\label{prop:antisource}
At the physical measure $\alpha=1$,
\begin{equation}\label{eq:antisource}
 \int \dot U_{\rm src}F^2\,\dd\mu_3
 =\int |\nabla\phi_{\rm src}|^2\,\dd\mu_3
 +\int_\R|\phi_{\rm src}(0,z)|^2\,\dd z
 \ge0.
\end{equation}
Equivalently, the instantaneous source contribution to
\[
 \mathcal K_2(F,U)=-\int UF^2\,\dd\mu_3
\]
is nonpositive.
\end{corollary}

\begin{remark}
The weighted theorem does not say that the full nonlinear evolution is anti-compressive.  It says something more precise: the instantaneous source leg $F^2\mapsto\partial_z(F^2)\mapsto G\mapsto U$ is non-amplifying for the compressive $F^2$ work throughout the Hardy-nonnegative range $1\le\alpha\le5$.  Any positive source-to-compression feedback must therefore involve pre-existing $G$ or subsequent transport, diffusion, or Hodge evolution.
\end{remark}

\section{Necessary conditions for singularity}

We now combine the two independent continuation mechanisms.  No concentration-selection, compactness, or profile hypothesis is used.

\begin{theorem}\label{thm:phenotype}
Let $u$ be the smooth axisymmetric solution with $u_0\in H^2(\R^3)$ and $\Gamma_0=ru_0^\theta\in L^\infty$, on its maximal interval $[0,T_*)$, with $T_*<\infty$ a hypothetical first singular time.  Then:
\begin{enumerate}[label=\rm(\roman*)]
\item For every sufficiently late $t_0<T_*$ and every finite $q\ge2$,
\begin{equation}\label{eq:necessary-action}
 \sup_{t_0<t<T_*}\int_{t_0}^t-\overline U_q(s)\,\dd s=\infty.
\end{equation}
More generally, \eqref{eq:necessary-action} holds for every $3/2\le q<2$ for which $Y_q(t_0)\in L^q$.  Thus the singularity must escape an entire continuum of scale-critical signed compression detectors.

\item For every sufficiently small terminal scale $h_0$, the integral Petrovskii condition \eqref{eq:petro-integral-NS} fails on that terminal neighborhood.  In particular, for every such $h_0$ there exist $\sigma\in(T_*-h_0,T_*)$ and $0<h<h_0$ such that
\begin{equation}\label{eq:necessary-petro}
 \int_{\sigma-h}^{\sigma}\norm{(u^r)_-(t)}_\infty\,\dd t
 >2\sqrt{\nu h\ell_{h_0}(h)}.
\end{equation}
More sharply, for every $\beta<1/2$, every $C_0\in\R$, and every sufficiently small $h_0$, some terminal sliding window violates
\begin{equation}\label{eq:necessary-petro-secondary}
 \frac{D_\sigma(h)^2}{4\nu h}
 \le
 \ell_{h_0}(h)
 +\beta\log\!\bigl(e+\ell_{h_0}(h)\bigr)+C_0.
\end{equation}

\item At every preterminal time, the physical Gateaux derivative in the pure instantaneous-source direction satisfies
\[
 \left.\frac{\dd}{\dd\varepsilon}\right|_{\varepsilon=0}
 \mathcal K_2(F,U+\varepsilon\dot U_{\rm src})\le0.
\]
Thus a positive source-to-compression feedback cannot be an instantaneous self-source effect.  The weighted extension in \cref{thm:weighted-antisource} holds whenever the corresponding power-weighted moments are finite, but such moments are not assumed here.
\end{enumerate}
\end{theorem}

\begin{proof}
Assertion (i) is \cref{cor:continuum-necessary-action}.  If the integral Petrovskii condition held on some terminal neighborhood, \cref{thm:petro-integral} would continue the solution.  If the leading wall \eqref{eq:necessary-petro} were never crossed on some sufficiently small terminal neighborhood, \cref{thm:petro-NS} would apply.  If \eqref{eq:necessary-petro-secondary} held uniformly for some $\beta<1/2$, $C_0$, and $h_0$, then \cref{cor:petro-secondary} would give continuation.  The contrapositives prove (ii).  Assertion (iii) is exactly the physical case \cref{prop:antisource} at each smooth preterminal time.
\end{proof}

Thus any positive source-to-compression return near a singular endpoint must use temporal evolution beyond the instantaneous source leg---for example pre-existing meridional vorticity, transport, diffusion, or a later phase of the Hodge response.

\begin{corollary}\label{cor:conditional-global}
Let $u$ be the maximal smooth axisymmetric solution generated by smooth $H^2$ axisymmetric data with bounded circulation.  Suppose that at every finite candidate terminal time $T$, at least one of the following holds on a terminal neighborhood:
\begin{enumerate}[label=\rm(\alph*)]
\item for some finite $q\ge2$ and some $0<t_0<T$, the critical signed action $\mathcal A_q(t_0,T)$ in \eqref{eq:critical-action-hyp} is finite;
\item the integral Petrovskii condition \eqref{eq:petro-integral-NS} holds; in particular it is enough that either the leading wall \eqref{eq:petroNS} or the lower-order condition \eqref{eq:petro-secondary} hold uniformly on sliding windows.
\end{enumerate}
Then the solution is global and smooth.
\end{corollary}

\section{Comparison with known criteria}

The weighted energies behind \cref{thm:two-parameter-action} belong to a classical line of axisymmetric analysis.  Chen--Fang--Zhang prove Prodi--Serrin criteria for $r^d u^\theta$, $0\le d<1$, and by interpolation with bounded circulation for $u^\theta/r^a$, $0<a\le1$ \cite{ChenFangZhang2017}.  Renc\l awowicz--Zaj\k aczkowski use the weighted circulation variable $u_\alpha=\Gamma/r^\alpha$ and select $\alpha s=3$ in their energy method \cite{RenclawowiczZajaczkowski2019}.  In our notation that relation is exactly
\[
 d=1-\frac3s,
\]
the $q=s$ critical diagonal in \eqref{eq:critical-Yq}.  Thus the critical weighted energies themselves are not claimed as new.  The new organization is the exact signed integrating factor and the resulting continuum of scale-invariant compression actions: instead of estimating $U=u^r/r$ by an unsigned norm, the criterion keeps the average of $U$ against the active critical swirl density and permits expansion to cancel compression.

Classical one-component criteria such as Chae--Lee \cite{ChaeLee2002} control a radial velocity or vorticity component through unsigned scale-critical space--time norms.  The present action is different in form: it retains only the signed average of $U$ against the evolving critical swirl density.  Pointwise, $-\overline U_q\le\|U_-\|_{L^\infty}$ whenever $Z_q>0$, but no general dominance relation with the broader component criteria is asserted.

The critical family also clarifies the endpoints.  At $q=3/2$, $Y_q=F=u^\theta/r$ and one recovers the strongest quotient detector of the earlier formulation.  At $q=3$, $Y_q=u^\theta$ and the critical mass is automatically finite from the ordinary Sobolev class; no bounded-circulation interpolation is needed for the continuation theorem.  As $q\to\infty$, $Y_q$ approaches the circulation $\Gamma=ru^\theta$, whose evolution has no radial stretching term.  The universal coefficient
\[
 q(1-d_q)=3
\]
shows that all finite points on the critical diagonal are measured by the same dimensionless signed action.

A related power-weight classification of the same swirl equation is developed in the author's preprint \cite{ShahmurovHardy2026}; there the emphasis is the sharp Hardy/coercivity phase diagram.  The present paper instead uses the exact evolution identity to obtain continuation from signed compression and combines it with an independent sliding-axis mechanism.

The Petrovskii criterion measures a different aspect of inward motion.  It is purely kinematic and one-sided: it controls the cumulative displacement generated by $(u^r)_-$ and uses the scalar maximum principle for $\Gamma$.  Zhang \cite{Zhang2026} proves regularity under the partial type-I hypothesis $u^r\gtrsim-(T-t)^{-1/2}$ by producing a H\"older circulation modulus.  The sliding condition here produces a weaker, non-H\"older modulus that nevertheless beats every negative power of $|\log r|$, sufficient by Wei's theorem.  It therefore permits a logarithmically larger averaged inward velocity.  It should also be compared with the pointwise radial criteria of Pan and Zhang \cite{Pan2017,ZujinZhang2018} and weighted radial criteria such as \cite{RenclawowiczZajaczkowski2019}.

The coefficient $2$ in \cref{thm:petro-NS} is the sharp leading Petrovskii coefficient, not asserted to be a necessary Navier--Stokes blowup constant.  The general integral criterion retains lower-order information: \cref{cor:petro-secondary} allows a positive $\log\log\log$ correction with coefficient $\beta<1/2$ while still producing a circulation modulus stronger than Wei's threshold.  The pulse construction in \cref{prop:terminal-no-go} shows that the sliding requirement is structural for this comparison route; a terminal-only cumulative bound does not control thin preterminal boundary layers.

Finally, \cref{thm:weighted-antisource,prop:antisource} explain why the instantaneous swirl source alone cannot supply the required compressive return.  The source $\partial_z(F^2)$ creates $G$, but its immediate Hodge response is non-amplifying for the power-weighted compressive $F^2$ work all the way to the quadratic Hardy endpoint $\alpha=5$.  If the swirl source participates in a positive source-to-compression feedback, temporal evolution must intervene between source creation and compressive return.  Together with \cref{thm:phenotype}, this leaves a hypothetical singularity with two simultaneous requirements: divergence of every automatically available critical signed action and recurrent escape from the sliding Petrovskii regime.

\section*{Conflict of interest}
On behalf of all authors, the corresponding author states that there is no conflict of interest.


\begin{thebibliography}{99}

\bibitem{Abdulla2005}
U.~G. Abdulla,
\newblock Multidimensional Kolmogorov--Petrovsky test for the boundary regularity and irregularity of solutions to the heat equation,
\newblock \emph{Boundary Value Problems} 2005 (2005), 181--199,
\newblock \url{https://doi.org/10.1155/BVP.2005.181}.

\bibitem{ChaeLee2002}
D.~Chae and J.~Lee,
\newblock On the regularity of the axisymmetric solutions of the Navier--Stokes equations,
\newblock \emph{Math. Z.} 239 (2002), 645--671,
\newblock \url{https://doi.org/10.1007/s002090100317}.

\bibitem{ChenFangZhang2017}
H.~Chen, D.~Fang and T.~Zhang,
\newblock Regularity of 3D axisymmetric Navier--Stokes equations,
\newblock \emph{Discrete Contin. Dyn. Syst.} 37 (2017), 1923--1939,
\newblock \url{https://doi.org/10.3934/dcds.2017081}.

\bibitem{ChenStrainTsaiYau2008}
C.-C.~Chen, R.~M.~Strain, T.-P.~Tsai and H.-T.~Yau,
\newblock Lower bound on the blow-up rate of the axisymmetric Navier--Stokes equations,
\newblock \emph{Int. Math. Res. Not. IMRN} 2008 (2008), Art. ID rnn016,
\newblock \url{https://doi.org/10.1093/imrn/rnn016}.

\bibitem{ChenStrainTsaiYau2009}
C.-C.~Chen, R.~M.~Strain, T.-P.~Tsai and H.-T.~Yau,
\newblock Lower bounds on the blow-up rate of the axisymmetric Navier--Stokes equations II,
\newblock \emph{Comm. Partial Differential Equations} 34 (2009), 203--232,
\newblock \url{https://doi.org/10.1080/03605300902793956}.

\bibitem{Ladyzhenskaya1968}
O.~A. Ladyzhenskaya,
\newblock \emph{The Mathematical Theory of Viscous Incompressible Flow},
\newblock 2nd ed., Gordon and Breach, New York, 1969.

\bibitem{LeiZhang2017}
Z.~Lei and Q.~S. Zhang,
\newblock Criticality of the axially symmetric Navier--Stokes equations,
\newblock \emph{Pacific J. Math.} 289 (2017), 169--187,
\newblock \url{https://doi.org/10.2140/pjm.2017.289.169}.

\bibitem{Muckenhoupt1972}
B.~Muckenhoupt,
\newblock Hardy's inequality with weights,
\newblock \emph{Studia Math.} 44 (1972), 31--38,
\newblock \url{https://doi.org/10.4064/sm-44-1-31-38}.

\bibitem{OzanskiPalasek2023}
W.~S. O\.{z}a\'nski and S.~Palasek,
\newblock Quantitative control of solutions to the axisymmetric Navier--Stokes equations in terms of the weak $L^3$ norm,
\newblock \emph{Ann. PDE} 9 (2023), Article 15,
\newblock \url{https://doi.org/10.1007/s40818-023-00156-7}.

\bibitem{Pan2017}
X.~Pan,
\newblock A regularity condition of 3D axisymmetric Navier--Stokes equations,
\newblock \emph{Acta Appl. Math.} 150 (2017), 103--109,
\newblock \url{https://doi.org/10.1007/s10440-017-0096-3}.

\bibitem{RenclawowiczZajaczkowski2019}
J.~Renc\l awowicz and W.~M. Zaj\k aczkowski,
\newblock On some regularity criteria for axisymmetric Navier--Stokes equations,
\newblock \emph{J. Math. Fluid Mech.} 21 (2019), Article 51,
\newblock \url{https://doi.org/10.1007/s00021-019-0447-0}.

\bibitem{UkhovskiiYudovich1968}
M.~R. Ukhovskii and V.~I. Yudovich,
\newblock Axially symmetric flows of ideal and viscous fluids filling the whole space,
\newblock \emph{J. Appl. Math. Mech.} 32 (1968), 52--61,
\newblock \url{https://doi.org/10.1016/0021-8928(68)90147-0}.

\bibitem{WangHuangWeiYu2024}
Y.~Wang, Y.~Huang, W.~Wei and H.~Yu,
\newblock Regularity criteria of the axisymmetric Navier--Stokes equations and Hardy--Sobolev inequality in mixed Lorentz spaces,
\newblock \emph{J. Math. Anal. Appl.} 533 (2024), Art. 128050,
\newblock \url{https://doi.org/10.1016/j.jmaa.2023.128050}.

\bibitem{Wei2016}
D.~Wei,
\newblock Regularity criterion to the axially symmetric Navier--Stokes equations,
\newblock \emph{J. Math. Anal. Appl.} 435 (2016), 402--413,
\newblock \url{https://doi.org/10.1016/j.jmaa.2015.09.088}.

\bibitem{Zhang2016}
Z.~Zhang,
\newblock Remarks on regularity criteria for the Navier--Stokes equations with axisymmetric data,
\newblock \emph{Ann. Polon. Math.} 117 (2016), 181--196,
\newblock \url{https://doi.org/10.4064/ap3856-3-2016}.

\bibitem{ZujinZhang2018}
Z.~Zhang,
\newblock A pointwise regularity criterion for axisymmetric Navier--Stokes system,
\newblock \emph{J. Math. Anal. Appl.} 461 (2018), 1--6,
\newblock \url{https://doi.org/10.1016/j.jmaa.2017.12.069}.

\bibitem{ShahmurovHardy2026}
R.~Shahmurov,
\newblock Hardy criticality in axisymmetric swirl,
\newblock arXiv:2605.01875 (2026).

\bibitem{Zhang2026}
Q.~S. Zhang,
\newblock On partial type I solutions to the axially symmetric Navier--Stokes equations,
\newblock arXiv:2604.07785 (2026).

\end{thebibliography}
\end{document}